\documentclass[a4paper,11pt]{article}%{report} 

\usepackage[affil-it]{authblk}
\usepackage{amsmath}
\usepackage{amsfonts}
\usepackage{hyperref}

\usepackage{amsthm}
\usepackage{amssymb}
\usepackage{graphicx}
\usepackage{tensor}
\DeclareGraphicsExtensions{.jpg,.png,.pdf,.eps,.bmp,.gif}
\newtheorem{theo}{Theorem}
\newtheorem{defi}{Definition}
\newtheorem{prop}[theo]{Proposition}
\newtheorem{lem}[theo]{Lemma}

\newtheorem{exa}[theo]{Example}
\newtheorem{rem}[theo]{Remark}

\title{IT formulae for gamma target: \\
mutual information and relative entropy} \author{Benjamin
  Arras\footnote{Universit\'e de Li\`ege, Sart-Tilman, All\'ee de la
    d\'ecouverte 12, B-8000 Li\`ege, Belgium
    \texttt{barras@ulg.ac.be}; BA's research is supported by a Welcome
    Grant from ULg. } \ and Yvik Swan\footnote{Universit\'e de
    Li\`ege, Sart-Tilman, All\'ee de la d\'ecouverte 12, B-8000
    Li\`ege, Belgium \texttt{yswan@ulg.ac.be}; YS gratefully
    acknowledges support from the IAP Research Network P7/06 of the
    Belgian State (Belgian Science Policy). }}  \affil{Universit\'e de
  Li\`ege}
  
  \date{}

\begin{document}
\maketitle

\begin{abstract}
\noindent 
In this paper, we introduce new Stein identities for gamma target distribution as well as a new non-linear channel specifically designed for gamma inputs. From these two ingredients, we derive an explicit and simple formula for the derivative of the input-output mutual information of this non-linear channel with respect to the channel quality parameter. This relation is reminiscent of the well-known link between the derivative of the input-output mutual information of additive Gaussian noise channel with respect to the signal-to-noise ratio and the minimum mean-square error. The proof relies on a rescaled version of De Bruijn identity for gamma target distribution together with a stochastic representation for the gamma specific Fisher information. Finally, we are able to derive precise bounds and asymptotics for the input-output mutual information of the non-linear channel with gamma inputs.

\


\end{abstract}

%Nouv 

\section{Introduction}
\label{sec:introduction}

% Copy paste from \cite[page 412]{Umbaugh}: the Gaussian model is most
% often used for natural noise processes. There exist other types of
% noise such as salt-and-pepper type noise (a.k.a.\ impulse noise, shot
% noise or spike noise) modeled by a uniform noise. Radar range and
% velocity images typically contain noise that can be modeled by the
% Rayleigh distribution and negative exponential noise occur in
% laser-based images. If this type of image is lowpass filtered the
% noise can be modeled as gamma noise. [...] Many of the types of noise
% that occur in natural phenomena can be modeled as some form of
% exponential noise.

% Other thing : while Guo Shamai Verdu's breakthrough is indepenedent of
% the distribution of the input, it rests crucially on the assumption
% that the noise is Gaussian. Make reference to their paper with general
% noise. 

% Gaussian Noise is the Standard model of
% amplifier noise.  Gamma Noise: Frequent in Radar/LiDAR.  Rayleigh
% Noise: One example where the Rayleigh distribution naturally arises is
% when wind velocity is analyzed into its orthogonal 2-dimensional
% vector components.  Exponential Noise: Channel Based Communication(see
% AEN). 

% Gaussian noise on gamma (or, more generally, positive) prior possibly
% does not make sense?

% Another intrinsic quantity is the second is given by 
% \begin{equation}
%   \label{eq:63}
%   I(X, N) = E \left[ \log \left( \frac{p(X, N)}{p(X) \gamma(N)} \right) \right]
% \end{equation}
% where $p(x, n)$ is the joint distribution of $X$ and $N$.

Let $X, Y$ be two random variables on the same probability space, with
joint probability measure $P^{X, Y}$ and marginals $P^X$ and $P^Y$,
respectively.  We choose the law of the couple $(X,Y)$ to be absolutely continuous
with respect to a common dominating measure $\mu$ and denote
$p^{X, Y}(x, y)$, $p^X(x)$ and $p^Y(y)$ the corresponding
Radon-Nikodym derivatives. The \emph{mutual information} between $X$
and $Y$ is
 \begin{equation}
  \label{eq:63}
  I(X; Y) = E \left[ \log \left( \frac{p^{X,Y}(X, Y)}{p^X(X) p^Y(Y)}
    \right) \right]. %= \mathrm{Ent}(X) + \mathrm{Ent}(Y) - \mathrm{Ent}(X, Y) 
\end{equation}
% where $\mathrm{Ent}(X)=- E \left[ \log p^X(X)\right]$ is the
% differential entropy of $X$ ($\mathrm{Ent}(Y)$ and $\mathrm{Ent}(X,Y)$
% are defined similarly). 
Mutual information satisfies $I(X; Y)\ge 0$ with equality if and only
if $X$ and $Y$ are independent and therefore mutual information
captures the dependence between $X$ and $Y$.  The \emph{relative
  entropy} (a.k.a.\ Kullback-Leibler divergence) from $Y$ to $X$ is
\begin{equation}
  \label{eq:13}
  D(X||Y) = E \left[  \log (p^X(X)/p^Y(X)) \right].
\end{equation}
Relative entropy satisfies $D(X || Y) \ge 0$ with equality if and only
if $X =^{\mathcal{L}} Y$ and therefore $D(X || Y)$ captures the
difference between $\mathcal{L}(X)$ and
$\mathcal{L}(Y)$. %(see e.g.\ \cite{Jo04, TC12}).
One speaks of Gaussian relative entropy if $Y$ is standard Gaussian.

Mutual information and relative entropy are crucial in a wide variety
of fields (see e.g.\ \cite{TC12} for an overview) but are both
generally analytically, and even in some cases algorithmically,
intractable. It is thus useful to dispose of formulas allowing to
control them in terms of quantities which are more amenable to
computations.  Two such formulas are Stam's \emph{De Bruijn identity}
\cite{Stam,Barr3} and Guo, Shamai and Verd\'u's \emph{MMSE identity}
\cite{GuoShamVer051} (which we shall refer to as GSV identity in the
sequel). Exact statements of these identities are deferred to
Section~\ref{sec:ident-gauss-targ}.  Informally, the De Bruijn
identity provides an explicit link between the Gaussian relative
entropy of an absolutely continuous random variable $X$ and the
\emph{Fisher information} of $X$. Similarly, the GSV identity provides
an explicit link between the mutual information in a Gaussian channel
and the \emph{minimal mean square error} in said channel. Both
formulas are, as it turns out, essentially equivalent because either
can be deduced - at least formally - from the other, see Sections \ref{sec:ident-gauss-targ} and
\ref{sec:form-mutu-inform}.  They relate information theoretic
quantities (relative entropy, mutual information) to quantities
typically of interest in statistical estimation theory (Fisher
information, MMSE) and have proven to be linchpins of important
developments in contemporary information theoretic probability theory
(e.g. for entropic CLTs \cite{Barr1, Jo04, Barr2, BBN, ABBN, Tulino}, analysis
of additive Gaussian channels \cite{LTV06,GuoShamVer051} or, more
generally, IT inequalities \cite{GuoShamVer06,Rioul}).

These two equivalent identities are inherently of a Gaussian nature
and it is therefore natural to enquire whether similar relationships
also hold outside of the Gaussian realm. Quoting \cite{GuoShamVer052},
``a natural question to pose is how general the information-estimation
relationship can be''. % Aside from their obvious theoretical
% interest, such relationships also are of strong practical interest  as
% applications of IT identities for Poisson, exponential, chi-square or
% Rayleigh, etc.\ targets abound in the literature. 
This important question has of course already received a lot of
attention in the literature and there exist De Bruijn identities, on
the one hand, and GSV identities, on the other hand, for most
classical target probability distributions of practical relevance
(precise references will be given later in the text).  The resulting
identities, however, no longer enjoy the elegance and ease of
manipulation as their Gaussian counterparts. In particular the
estimation quantities derived in either cases do not bear natural
interpretations and, to the best of our knowledge, the equivalence
between the general-target De Bruijn and the general-target GSV
identities has never been investigated.  As is evident from an
inspection of their proofs, both the De Bruijn and the GSV identities
are obtained through a study of entropy/information jumps around $X$
along small perturbations of the form
\begin{equation}
  \label{eq:67}
X \mapsto  X_{r} :=   \sqrt{r} X + N
\end{equation}
with $r>0$ and $N$ an independent standard Gaussian. 
% Both De Bruijn's and the GSV identity are of an intrinsically Gaussian
% nature. 
A first intuitive way to branch outside of the Gaussian scope is to
work as in \cite{GuoShamVer052,PV07} and extend \eqref{eq:67} by
considering more general (additive or even non-additive) noising
mechanisms of the form $X \mapsto X_a=h(a; X; W)$ where $a$ is a real
parameter, $h(\cdot; \cdot; \cdot)$ is a deterministic function and
$W$ is an independent noise following some arbitrary distribution.  As
could be expected, the classical MMSE quantity from estimation theory
no longer plays any role in these identities, and the corresponding
correct object is expressed as the correlation of two generally
intractable conditional expectations (depending on log-derivatives of
the density of $X_{r}$) which bears no explicit representation nor
interpretation.

Now, depending on the context, the perturbation $X_{r}$ in
\eqref{eq:67} is referred to as an ``additive Gaussian channel''
\cite{TC12}, a ``smart path'' \cite{NP3} or an ``Ornstein-Uhlenbeck
evolute'' around $X$ \cite{BGL}. The key fact here is that the
deformation $x \mapsto \sqrt{r} x + N$ arises naturally through the
action of the Ornstein-Uhlenbeck semigroup and thus $X_{r}$ ought
to be interpreted as a stochastic representation for the ``smart
path'' interpolation between the law of $X$ and the Gaussian
distribution. A general take on this semi-group interpretation leads
to De Bruijn-type formulas with general reference probability measure
(see \cite{BE,BGL}) providing a direct link between the relative
entropy $D(X \, || \, Z)$ from a target random variable $Z$ to a
random variable $X$ and a \emph{target-specific} Fisher information
structure. This Fisher information structure is, in general, implicit
as it depends on the distribution of the \emph{ad hoc} deformation
$X_{r}$ which bears no explicit stochastic representation
equivalent to \eqref{eq:67}.  

In this paper we derive a new set of De Bruijn/GSV identities
specifically when the target distribution is in the family of gamma
distributions (which encompass as particular cases the chi-square and
exponential distributions). There are two main new ingredients behind
our results. The first ingredient is a family of \emph{Stein
  identities} for gamma target distribution. Stein identities are
characterizations of probability distributions through the action of
target-specific differential operators (see e.g.\ the Gaussian Stein
identity \eqref{eq:1}). They are available for virtually any
probability distribution allowing a closed form distribution
(\cite{LRS}) and are known to lie at the boundary between IT and
estimation theory \cite{BDH06,PSQ12,LSa,NPS2,NPS}. The second
ingredient is a new noising-channel $r \mapsto X_r$ specifically
designed for gamma input (see \eqref{eq:46}); interestingly this
channel is quadratic rather than linear as in \eqref{eq:67}.  By
combining these two concepts we derive, via elementary arguments,
tractable and interpretable gamma-specific De Bruijn and GSV
identities.  While the De Bruijn identity is in essence a rescaling of
known results from our previous paper \cite{ArSw1}, the gamma-GSV
identity we obtain is entirely new. We prove that our quadratic
channel has properties which are strikingly similar to the additive
Gaussian channel in terms of mutual information and its asymptotics
for large values of the channel quality parameter.
% An important byproduct
% of our approach is that we introduce a new intrinsic gamma noising
% channel (see \eqref{eq:46}) which turns out to be quadratic rather
% than linear as in \eqref{eq:67};  
% In particular we also provide new intrinsic definitions of
% gamma-specific Fisher information andquantity from estimation theory.

\subsection{Outline of the paper}
In Section \ref{sec:ident-gauss-targ} we review the relevant known
results for Gaussian target. In Section
\ref{sec:it-stein-identities-1} we provide the necessary IT and Stein
identities for gamma target and we also recall the \emph{ad hoc}
gamma-specific De Bruijn identity (Theorem
\ref{sec:it-stein-identities-2}). In Section
\ref{sec:nonl-gamma-chann} we discuss the main properties of the
gamma-counterpart to the smart path \eqref{eq:67} and in Section
\ref{sec:interr-gamma-mmse} (mainly
Proposition~\ref{prop:interr-gamma-mmse-1}) we provide the key
ingredient of the paper, namely a new representation of the
(gamma-specific) Fisher information in terms of a quantity reminiscent
of the minimal mean square error at the heart of  the  GSV equality. 
In Section \ref{sec:form-mutu-inform} we show that the quantities we
have introduced are indeed the missing link between IT and estimation
theory with gamma target: 
% quantity is indeed the relevant for estimation o study the variation of the mutual information with
% respect to the signal-to-noise ratio through the additive Gaussian
% channel and the quadratic gamma channel. % It is now a classical result
% % that the derivative of the mutual information with respect to the
% % signal-to-noise ratio in the Gaussian case is given by the minimum
% % mean-square error (see \cite{GuoShamVer051}) under a finite second
% % moment assumption for the input distribution.
% We derive an analytical
% proof of this fact through a simple link between mutual information
% and relative entropy. This proof combines ideas contained in Section
% II-D of \cite{GuoShamVer051} together with a relation similar to
% formula (65) of \cite{GuoShamVer051}. The method behind this proof is
% robust towards a change of channel, in the sense that it is readily
% transposed from the additive Gaussian channel to the quadratic gamma
% channel. Moreover, for the gamma channel, we
we derive an explicit GSV formula for gamma target as well as fine
upper bounds for the variation of the mutual information with respect
to the channel quality parameter. The bounds are universal to the extent
that they depend on the distribution of the input only through its
mean and the estimation theoretical quantity put forward in
Proposition~\ref{prop:interr-gamma-mmse-1}. The only assumption needed
on the input $X$ is the existence of finite $\alpha+4$
moment. Finally, for gamma input with parameters $(\alpha,\lambda)$,
we obtain an inequality on the mutual information reminiscent of the
Gaussian case and for $\alpha=1/2$, we obtain the exact asymptotic 
for large values of the channel quality parameter of the input-output mutual information.

\section{IT and Stein Identities for Gaussian target}
\label{sec:ident-gauss-targ}

Let $N$ be a standard Gaussian random variable with pdf
$\gamma(x) = (2\pi)^{-1/2} e^{-x^2/2}$. Stein's well-known identity \cite{S81,Stein1986}
states that:
\begin{equation}
  \label{eq:1}
  E \left[ N \phi(N)  \right] =  E \left[ \phi'(N) \right]
  \mbox{ for all } \phi \in  \mathcal{F}(N)
\end{equation}
with $\mathcal{F}(N)$ the collection of absolutely continuous test
functions $\phi : \mathbb{R} \to \mathbb{R}$ such that
$\phi' \in L^1(N)$.  Moreover if another random variable $X$ also
satisfies \eqref{eq:1} then $X {=}^{{\mathcal{L}}}{Z}$. We refer the
reader to \cite[Lemma 3.1.2]{NP3} for a streamlined proof.
Extending identity \eqref{eq:1} to arbitrary target entices us to
associate to any random 
variable $X$  with
mean $\mu$ and variance $\sigma^2$ a random variable  $\rho_X(X)$ defined
(almost everywhere) through the identity:
\begin{align}
  \label{eq:2}
  & E \left[ \rho_X(X) \phi(X)  \right] = -  E \left[\phi'(X) \right]
    \mbox{ for all } \phi \in \mathcal{F}(X)
\end{align}
with $\mathcal{F}(X)$ the collection of absolutely continuous test
functions $\phi : \mathbb{R} \to \mathbb{R}$ such that
$\phi' \in L^1(X)$.  The random variable $\rho_X(X)$ defined a.e.\ by
\eqref{eq:2} is called the score of $X$; it is easy to see that if $X$
has differentiable density $p_X$ which cancels at the border of its
support then
$\rho_X(X) = \frac{\mathrm{d}}{\mathrm{d}x}\log p_X(x) \, |_{x = X}$
satisfies \eqref{eq:2}.  In particular from \eqref{eq:1} we know that
$ \rho_X(X) =- \frac{X-\mu}{\sigma^2} $ if and only if
$X {=}^{{\mathcal{L}}} \sigma N + \mu$ (here and throughout we reserve
the notation $N$ for a standard normal random variable).  Conditions
on the distribution of $X$ under which the score is well-defined have
been thoroughly adressed in the literature and it is a well-known fact
that the score is essentially unique in the sense that if a random
variable $Y$ satisfies \eqref{eq:2} with the same score as $X$ then
$Y =^{\mathcal{L}}X$; see \cite{Stein1986,Jo04, LRS}

Applying \eqref{eq:2} to the test function $\phi(x) = 1$ we deduce
that if $X$ admits a score then necessarily
$E[\rho_X(X)] = 0$.  The second moment of the score plays a
role in the \emph{standardized relative Fisher information}
\begin{align}
  \label{eq:4}
   J_{\mathrm{st}}(X) = \sigma^2E \left[ (\rho_X(X) + (X-\mu)/\sigma^2)^2 \right]
    = \sigma^2E \left[ \rho_X(X)^2 \right] - 1
\end{align}
and it is well known that $J_{\mathrm{st}}(X) = 0$ if and only if
$X =^{\mathcal{L}} \sigma N + \mu$ (see e.g.\ \cite{Jo04,Barr2,NPS2});
in other words the second moment of the score suffices to characterize
the distribution. The quantity $I(X) = E[\rho_X(X)^2]$ is
the \emph{Fisher information} of $X$ and,   from previous considerations we
know that $I( \sigma N + \mu) = \frac{1}{\sigma^2}$. %For simplicity
                                                     %we
% set $\mbox{Var}(X) = E[X^2]=\sigma^2 = 1$ for the rest of this
% section. 

Relative entropy
\eqref{eq:13} and standardized Fisher information~\eqref{eq:4} are
related through the classical De Bruijn identity
\begin{align}
  \label{eq:11}
  \frac{\mathrm{d}}{\mathrm{d}r} D(X_{r}||N) & = \frac{1}{2r}\bigg(I(X_r)-1+r\bigg)
 , \\
  \label{eq:DeBruijn}
  & = \frac{1}{2(1+r)}\bigg(r+\frac{1}{r}J_{st}(X_r)\bigg),
\end{align}
still with $X_{r}$ as in \eqref{eq:67} (see \cite{Barr3,Jo04} for a
proof of \eqref{eq:DeBruijn} solely under moment assumptions on $X$).
Applying a conditional version of \eqref{eq:1}, we note how for
all sufficiently regular test functions $\phi$ we also have
$$E \left[ (E[\sqrt{r} X \, | \, X_r] - X_r) \phi(X_r)
\right] = - E \left[ N \phi(X_r) \right] = -E \left[
  \phi'(X_{r}) \right]$$ from which we deduce the representation
\begin{equation}
  \label{eq:10}
  \rho_r(X_r) =\sqrt{r} E[ X \, | \, X_r] -  X_r
\end{equation}
for $\rho_r(X_r)$ the score of $X_r$. This in turn leads to
the representation of standardized Fisher information:
\begin{align}%\label{eq:14}
% I(X_{r}) & =  r \left( E \left[  E \left[ X \, | \, X_{r} \right]^2
%                 \right] -1 \right)+ 1 = 1 - r \mathrm{MMSE}(X,
%               X_{r}) \\
\frac{1}{r}J_{\mathrm{st}}(X_{r}) & =  1- (1+r)
                           E \left[ \left( X- E[X \, | \, X_{r}]
                                          \right)^2
                                    \right]         \label{eq:56} 
\end{align}
which provides a connection between Fisher information (and hence
relative entropy) with estimation theory's Minimal Mean Square Error
\begin{equation}
  \label{eq:7}
  \mathrm{MMSE}(X, Y) =  E \left[ \left( X- E[X \, | \, Y]
    \right)^2 \right]. 
\end{equation} 
Plugging \eqref{eq:56} into \eqref{eq:11} we obtain
\begin{align}
  \label{eq:9}
  \frac{\mathrm{d}}{\mathrm{d}r} D(\sqrt{r}X + N||N)  = \frac{1}{2}
  \left( 1- \mathrm{MMSE}(X, X_{r}) \right),
\end{align}
which is equivalent to formula $(66)$ of Theorem $5$ page $7$ of \cite{GuoShamVer051}.

Let $X$ be centered with finite variance. As already touched upon in
the introduction, the original GSV formula from \cite{GuoShamVer051}
links mutual information \eqref{eq:63} and MMSE \eqref{eq:7} through:
\begin{equation}
  \label{eq:MMSE}
  \frac{\mathrm{d}}{\mathrm{d}r} I(X; X_{r}) = \frac{1}{2}
  \mathrm{MMSE}(X, 
  X_{r}). 
\end{equation}
We conclude this section by showing how to obtain \eqref{eq:MMSE} from
\eqref{eq:9}; our argument relies on ideas from Section II-D of
\cite{GuoShamVer051}. We stress that our method of proof is robust
towards a change of channel, in the sense that we will show in Section
\ref{sec:form-mutu-inform} how it can be transposed from the Gaussian
to the gamma setting.  For $r>0$ we first set $\tau(r)=r/(r+1)$ and
introduce the random variables
\begin{align*}
  &\tilde{X}_{\tau(r)}(x)=\sqrt{\tau(r)}x+\sqrt{1-\tau(r)}N\\
  & X_r(x) = \sqrt{r}x + N
\end{align*}
with $x \in (-\infty, +\infty)$ and $N$, as above, an independent
standard Gaussian. We also set $X_r = X_r(X)$ and $\tilde{X}_{\tau(r)}
= \tilde{X}_{\tau(r)}(X)$. For any deterministic functional, we denote by 
$E_X[F(\tilde{X}_{\tau(r)}(X))]$ (similarly with $X_r(X)$) the 
following type of integral:
\begin{align}\label{ExpectationX}
E_X[F(\tilde{X}_{\tau(r)}(X))]=\int p_X(x)F(\tilde{X}_{\tau(r)}(x))dx.
\end{align}
By standard  arguments we know that
$I(X, \tilde{X}_{\tau(r)}) = I(X, X_r)$ and 
\begin{equation}
\label{MI-RE-gau}
I(X,\tilde{X}_{\tau(r)})= E_X \left[  D(\tilde{X}_{\tau(r)}(X)\mid\mid N) \right]-D(\tilde{X}_{\tau(r)}\mid\mid N),
\end{equation}
where $E_X[\cdot]$ denotes an expectation taken with respect to $X$.
%(see e.g.\ forthcoming Lemma \ref{fromMItoRE} for a proof).
% and $\tilde{X}_{\tau(r),\beta,x}$ matches $(43)$ and $(44)$ where
% the law of $X_{\beta}$ is the Dirac mass at
% $x$.
Using regularity  arguments provided in \cite{Barr3} in combination
with the chain rule we easily obtain:
\begin{equation}
\label{DB-MMSE-eq-gau}
\dfrac{d}{dr}\bigg(I(X,X_{r})\bigg)=\dfrac{1}{r(r+1)}\bigg(E_X
\left[ J_{st}(\tilde{X}_{\tau(r)}(X)) \right]-J_{st}(\tilde{X}_{\tau(r)})\bigg).
\end{equation}
We can finally conclude.

\begin{prop}\label{GSVGauss}[Theorem $1$ in \cite{GuoShamVer051}] 
Identity \eqref{eq:MMSE} holds if $X$ is centered with
  ${E}[X^2]=1$. 
\end{prop}

\begin{proof}
  First note how for all $x$ the random variable
  $\tilde{X}_{\tau(r)}(x)$ remains Gaussian so that straightforward
  computations lead to:
\begin{align*}
J_{st}(\tilde{X}_{\tau(r)}(x))=\frac{1}{2}\dfrac{r(r+x^2)}{1+r}, 
\end{align*}
for all real $x$.  Next,  by scaling arguments, we get
\begin{align}
  J_{st}(\tilde{X}_{\tau(r)})=(1+r)J_{st}(X_{r})-\frac{r^2}{2}.
\end{align}
Moreover, using \eqref{eq:56} we obtain
\begin{align}
  J_{st}(\tilde{X}_{\tau(r)})&=\frac{1}{2}\bigg[(1+r)r(1-\operatorname{MMSE}\big(X
                                                      ,X_{r}\big))-r^2\bigg],\\
                                                    &=\frac{1}{2}\bigg[-r(1+r)\operatorname{MMSE}\big(X,X_{r}\big)+r\bigg]
\end{align}
Combining everything together, we have:
\begin{align*}
  \dfrac{d}{dr}\bigg(I(X,X_{r})\bigg) % =\dfrac{1}{2r(r+1)}\bigg(\int
                                         %              p_X(x)J_{st}(\tilde{X}_{\tau(r)}(x))dx-J_{st}(\tilde{X}_{\tau(r)}(X))\bigg),\\
&=\dfrac{1}{2r(r+1)}\bigg(
  E \left[ \dfrac{r(r+X^2)}{1+r} \right]+r(1+r)\operatorname{MMSE}\big(X,X_{r}\big)-r\bigg),\\
&=\frac{1}{2}\operatorname{MMSE}\big(X,X_{r}\big),
\end{align*}
as required.
\end{proof}

\section{IT and Stein identities for gamma target}
\label{sec:it-stein-identities-1}

Let $Z$ be a gamma distributed random variable with pdf
$\gamma_{\alpha, \lambda}(x) = \lambda^{\alpha}/\Gamma(\alpha)
x^{\alpha-1} \mathrm{exp}(-\lambda x)$
over the positive half line. The equivalent of Stein's identity
\eqref{eq:1} for a gamma target has long been known to be
\begin{equation}
  \label{eq:29}
  E \left[ (\lambda Z-\alpha ) \phi(Z) \right] = E \left[Z \phi'(Z)
\right]
\end{equation}
(see \cite{LukPhD}). Moreover if some positive random variable $X$
also satisfies this identity over an appropriately wide class of
functions then $X =^{\mathcal{L}} Z$, see \cite{GaAlRe15} for a proof.
Introducing the derivative $\partial_x^{\sigma}\phi(x) =
(\sqrt{x}\phi(x))'$,  we rewrite \eqref{eq:29} as 
\begin{equation}
  \label{eq:15}
  E \left[  \sqrt{Z}(\lambda \sqrt{Z}-(\alpha - {1}/{2})/{\sqrt{Z}}) \phi(Z)\right] = E
  \left[ \sqrt{Z}\partial_x^{\sigma}\phi(Z) \right]  \mbox{ for all }\phi \in
  \mathcal{F}(Z). 
\end{equation} 
with $\mathcal{F}(Z)$ a collection of sufficiently smooth test function $\phi : \mathbb{R}_+^* \to \mathbb{R}$ such that
$x \mapsto \sqrt{x}\partial_x^{\sigma}\phi(x) \in L^1(Z)$. 
While in appearance less elegant than \eqref{eq:29}, we claim that 
 \eqref{eq:15} is actually the correct starting point for
Stein/IT analysis with a gamma target. 

As in Section \ref{sec:ident-gauss-targ} we begin by extending the
scope of \eqref{eq:15} to arbitrary target by introducing for
arbitrary positive $X$ a random variable $\rho^{\gamma}_X(X)$ defined
(almost everywhere) through the identity:
\begin{align}
     E \left[  \sqrt{X}\rho_X^{\gamma}(X) \phi(X)\right] = - E
  \left[ \sqrt{X} \partial_x^{\sigma}\phi(X) \right]  \mbox{ for all }\phi \in
  \mathcal{F}^{\sigma}(X) \label{eq:8}
\end{align}
with $\mathcal{F}^{\sigma}(X)$ the collection of absolutely continuous
test functions $\phi : \mathbb{R} \to \mathbb{R}$ such that
$x \mapsto \sqrt{x}\partial_x^{\sigma}\phi(x) \in L^1(X)$. We call
$\rho_X^{\gamma}(X)$ defined by \eqref{eq:8} $X$'s $\gamma$-score.
Taking $\phi(x) = 1/\sqrt{x}$ in \eqref{eq:8} we conclude that if $X$
admits a $\gamma$-score then necessarily it satisfies
$E[\rho_X^{\gamma}(X)] =0$.  From \eqref{eq:15} we know that the
$\gamma(\alpha, \lambda)$ distribution is characterized by
\begin{align}
  \label{eq:27}
  \rho_Z^{\gamma}(Z) = -(\lambda \sqrt{Z}-(\alpha - {1}/{2})/{\sqrt{Z}})
\end{align}
Mimicking the Gaussian situation from Section
\ref{sec:ident-gauss-targ} it is natural to measure distance to the
gamma by comparing $\gamma$-scores with those in \eqref{eq:27}.

\begin{defi}[Standardized gamma Fisher information]
  The   standardized $\gamma(\alpha, \lambda)$-Fisher information of
 a  positive random variable $X$ with finite mean and pdf $p$ is:
\begin{align}
\label{eq:37}
  J_{\mathrm{st}, \gamma(\alpha, \lambda)}(X) = \frac{1}{\lambda}E
  \left[ \left(\rho^{\gamma}_X(X) + \lambda \sqrt{X} -\frac{\alpha -
  {1}/{2}}{\sqrt{X}} \right)^2  \right]. 
\end{align}
\end{defi}
Standardized gamma Fisher information is not location invariant (we
need the input to be positive) but behaves nicely under scaling (under the assumption that $E[X]=\alpha/\lambda$): 
\begin{equation}
  \label{eq:12}
  J_{\mathrm{st}, \gamma(\alpha, \lambda)}(aX) = J_{\mathrm{st},
    \gamma(\alpha, a\lambda)}(X) = \frac{1}{a}J_{\mathrm{st},
    \gamma(\alpha, \lambda)}(X) + \alpha \frac{(a-1)^2}{a}. 
\end{equation}
  Note that (by straightforward integration by parts starting from
\eqref{eq:8})
\begin{align}
  \label{eq:22}
  \sqrt{X}\rho_X^{\gamma}(X) =  
  X \rho_X(X) + \frac{1}{2}  
\end{align}
with $\rho_X(x) = (\log p_X(x))'$ the usual score of $X$ (here we abuse
notations slightly w.r.t.\ the definitions from Section
\ref{sec:ident-gauss-targ}).  Hence we can rewrite
\eqref{eq:37} as
\begin{equation}
  \label{eq:58}
  J_{\mathrm{st}, \gamma(\alpha, \lambda)}(X) = \frac{1}{\lambda} E \left[ X (\rho_X(X)
  + \lambda -  (\alpha-1)/X)^2  \right]
\end{equation}
which is precisely the relative Fisher information advocated by
\cite{BGL}.  Aiming at a Cramer-Rao inequality one might wish to
expand the square in \eqref{eq:37} in order to
identify the correct gamma-Fisher information, but it is easy to
realize that this will not yield good results. Following \cite{ArSw1}
we rather propose to introduce
 \begin{equation}
   \label{eq:31}
   I_{\gamma(\alpha, \lambda)}^{r}(X) =  \frac{1}{\lambda} E \left[ X
     \left(\rho_X(X)+  \lambda(1+r)-\frac{(\alpha-1)}{X}\right)^2 \right]. 
 \end{equation}
 which we call a $r$-corrected gamma Fisher information.  Clearly
 $I_{\gamma(\alpha, \lambda)}^r(Z) = \alpha r^2$ for all $r \ge 0$ and all
 $\lambda>0$ if $Z \sim \gamma_{\alpha, \lambda}$ (recall that
 $\rho_Z(Z) = (\alpha-1)/Z - \lambda$ in this case) and simple
 computations show that
\begin{equation}
  \label{eq:32}
  J_{\mathrm{st}, \gamma(\alpha, \lambda)}(X) =  I_{\gamma(\alpha, \lambda)}^{r}(X) - \alpha
  r^2 \ge 0
\end{equation}
(we stress the important fact that this decomposition holds solely
under a first moment assumption on $X$, see also \cite[Remark
13]{ArSw1}).

The relative entropy with respect to the gamma distribution is defined
exactly as in the Gaussian case (recall \eqref{eq:13}):
\begin{equation}
  \label{eq:35}
  D(X || \gamma(\alpha, \lambda)) = \int_0^{+\infty} p_X(u)
  \log(p_X(u)/\gamma_{\alpha,\lambda}(u)) du
\end{equation}
with $X$ a random variable with density $p_X$ on the positive real
line. Note how gamma relative entropy does not behave as Gaussian
relative entropy under scaling:  
\begin{equation}
  \label{eq:52}
  D(aX || \gamma(\alpha, \lambda)) = D(X || \gamma(\alpha, a\lambda))
\end{equation}
for all $a>0$. 
There exists a De Bruijn identity specifically for
\eqref{eq:35}, first identified by \cite{BE,BGL} in the context of
probability semigroup theory and $\Gamma$-calculus. We state a
rescaling of the identity in its most general form as due to
\cite{ArSw1}.

\begin{theo}[Gamma De Bruijn identity]\label{sec:it-stein-identities-2}
Let  $\alpha\ge1/2$ and suppose that $X$ is a random variable with
finite $\alpha+4$ moments.  % This quantity is translation
%  invariant and  satisfies
%  \begin{equation}
%    \label{eq:36}
%    D(aX || \gamma) = D(X || \gamma) - \log a
%  \end{equation}
% for all $a>0$.
Then
\begin{equation}
  \label{eq:33}
  \frac{\mathrm{d}}{\mathrm{d}r} D (X_r || \gamma(\alpha, \lambda/(1+r))) = \frac{1}{r}
   J_{st, \gamma(\alpha, \lambda)}(X_{r}) - \alpha \frac{r}{1+r}
\end{equation}
where
\begin{equation}
  \label{eq:34}
  X_{r} =  \gamma(\alpha-\frac{1}{2}, \lambda) + \left(
    \sqrt{r X} + \frac{N}{\sqrt{2\lambda}} \right)^2,
\end{equation}
with $\gamma(\alpha-\frac{1}{2}, \lambda)$ an independent gamma
distributed random variable with parameters $\alpha-1/2, \lambda$ and
$N$ as before an independent standard Gaussian random variable. The
integrated version is 
\begin{equation}
  \label{eq:59}
  D(X \, || \, \gamma(\alpha, \lambda)) = \int_0^{\infty}\left(  \frac{1}{r}
   J_{st, \gamma(\alpha, \lambda)}(X_{r}) - \alpha \frac{r}{1+r}\right)dr.
\end{equation}
\end{theo}
  
\section{A quadratic gamma channel}
\label{sec:nonl-gamma-chann}

Equations \eqref{eq:33} and \eqref{eq:34} lead us to introducing the
nonlinear gamma channel (with all notations as in Theorem
\ref{sec:it-stein-identities-2})
\begin{equation}
  \label{eq:46}
  X \mapsto X_r (:= X_{r, \alpha,\lambda}) = \gamma(\alpha-1/2, \lambda) + \left( \sqrt{rX} +
    \frac{N}{\sqrt{2\lambda}} \right)^2
\end{equation}
for $r>0$.  We also introduce the notation
\begin{equation}
  \label{eq:16}
Y_{r} = \sqrt{rX} + \frac{N}{\sqrt{2\lambda}}
\end{equation}
Conditionally on $X$, the random variable $X_r$ is the independent sum
of a gamma and a non-central chi-squared random variable. This is the
main difference between our channel \eqref{eq:46} and classical
``dual'' channels wherein the distribution of the output,
conditionally on the signal, remains within the same family of
distributions as the noise (such as for instance in Gaussian channels
studied in Section \ref{sec:ident-gauss-targ} or Poisson channels
\cite{GuoShamVer08}). 

 Exploiting the moment generating function of $X_r$ we
obtain the following description of the channel.
\begin{prop}\label{prop:someproperties}
  \begin{itemize}
  \item If $X$ has moment generating function $M_X(\cdot)$ on $(0, a)$
    then the moment generating function of $X_r$ is
  \begin{equation} 
    \label{eq:50}
    M_r(t) = \left( 1-\frac{t}{\lambda}
  \right)^{-\alpha} M_X \left( \frac{ rt }{1-\frac{t}{\lambda}} \right)
  \end{equation}
on $(0,  \lambda/(\lambda r/a + 1))$. 
\item In particular if  $E[X] = \alpha/\lambda$ then 
  \begin{equation*}
    E[X_r] = \frac{\alpha}{\lambda} + r E[X] = \frac{\alpha}{\lambda}(1+r).
  \end{equation*}
\item   Let
$\lambda_r \le \lambda$.  The
output $X_r$ is itself gamma distributed with parameters $(\alpha,
\lambda_r)$  if and only if the input is gamma distributed with
parameters $(\alpha, \frac{1}{r}
(\frac{1}{\lambda_r}-\frac{1}{\lambda}))$. 
  \end{itemize}
\end{prop}
\begin{proof}
  Identity \eqref{eq:50} follows by independence as well as the fact
  that,  conditionally  on $X$, the
random variable  $\left( \sqrt{2\lambda rX} +
    N \right)^2$ is noncentral
  chi square distributed with non-centrality parameter
  $\sqrt{2\lambda rX}$. Hence 
  \begin{align*}
   E \left[ e^{tX_r}  \right]  &  =  \left( 1-\frac{t}{\lambda}
  \right)^{-(\alpha-1/2)} E \left[ e^{t\left( \sqrt{rX} +
        \frac{N}{\sqrt{2\lambda}} \right)^2} \right] \\
& = \left( 1-\frac{t}{\lambda}
  \right)^{-(\alpha-1/2)} \frac{E \left[ e^{ \frac{ rt }{1-t/\lambda}X}  \right]  }{(1-t/\lambda)^{1/2}},
  \end{align*}
  which is defined as long as $t\le \lambda$ and
  ${ rt }/({1-t/\lambda})\le a$.  To see the next claim it suffices to
  notice that if $M_r(t) =(1-t/\lambda_r)^{-\alpha}$ then necessarily
\begin{equation*}
  M_X(t) = \left( 1 - \frac{t}{r} \left(\frac{1}{\lambda_r} -
      \frac{1}{\lambda}  \right) \right)^{-\alpha} 
\end{equation*}
for $t$ sufficiently small. 
\end{proof}
\begin{rem}\label{rem:gammagamma}
  An equivalent way to express the second point in
  Proposition~\ref{prop:someproperties}: if $X$ is gamma distributed
  with parameters $\alpha, \lambda_1$ then $X_r$ is gamma distributed
  with parameters $(\alpha, \lambda_r)$ where
  $\lambda_r = (\frac{r}{\lambda_1} + \frac{1}{\lambda})^{-1}$ for all
  $r>0$.
\end{rem}
% We  compute the relative entropy between $X$ and $X_r$: 
%   \begin{equation}
%     \label{eq:49}
%     D(X || X_r) = \alpha \left( 
%     \log \left(  \frac{\lambda_1}{\lambda_r} \right) +  \frac{\lambda_1}{\lambda_r}-1\right)
%   \end{equation}
% \begin{prop}
% The   capacity of  nonlinear channel \eqref{eq:46}
% is 
% \begin{equation}
%   \label{eq:45}
%   C  = \sup_X I(X, X_r) = 
% \end{equation}
% where the supremum is taken over all possible positive random
% variables $X$ with mean $\alpha/\lambda$. 
% \end{prop}

\section{Relative entropy and estimation theory}
\label{sec:interr-gamma-mmse}

\begin{prop}\label{prop:interr-gamma-mmse-1}
  Let $\alpha\ge 1/2$ and suppose that $X$ is positive with finite
  mean. Define $X_r, Y_r$ as in \eqref{eq:46}, \eqref{eq:16} and
  introduce the ratio
  \begin{equation}
    \label{eq:54}
  \operatorname{\mathcal{V}}_r(X) = \frac{Y_r}{\sqrt{X_r}}.  
  \end{equation}
Then
  \begin{equation}
    \label{eq:38}
    \rho_r^{\gamma}(X_r) + \lambda \sqrt{X_r} -(\alpha- 1/2)/\sqrt{X_r} = E \left[ \lambda \sqrt{r X}
    \operatorname{\mathcal{V}}_r(X)  \, | \, X_r \right]
  \end{equation}
and  
\begin{equation}
  \label{eq:25}
  J_{st,\gamma(\alpha, \lambda)}(X_r)  = \lambda E \left[ E \left[ \sqrt{r X}
     \operatorname{\mathcal{V}}_r(X)  \, | \, X_r \right]^2 \right].
\end{equation}
\end{prop}

\begin{rem}
  Note how in particular if $\alpha=1/2$ then \eqref{eq:54} reduces to
  $\mathrm{sign}(Y_r)$, the sign of $\sqrt{rX} +
  N/\sqrt{2\lambda}$. This quantity plays a central
  role in Stein type representations for gamma specific Fisher information as obtained in
  \cite[Proposition 23]{ArSw1}.
\end{rem}
% \subsection{The case $\alpha=1/2$}
% We start with an equivalent to \eqref{eq:10} in the case
% $\alpha = 1/2$. 
% \begin{prop}\label{prop:it-stein-identities}
%   Set $\alpha= 1/2$ and
%   $X_r = (\sqrt{r X} + \frac{N}{\sqrt{2\lambda}})^2 =: Y_r^2$. Then
%   \begin{equation}
%     \label{eq:38}
%     \rho_r^{\gamma}(X_r) + \lambda \sqrt{X_r} = E \left[ \lambda \sqrt{r X}
%       \mbox{sign}(Y_r) \, | \, X_r \right]
%   \end{equation}
% and  
% \begin{equation}
%   \label{eq:25}
%   J_{st,\gamma}(X_r)  = \lambda E \left[ E \left[ \sqrt{r X}
%       \mbox{sign}(Y_r) \, | \, X_r \right]^2 \right].
% \end{equation}
% \end{prop}
\begin{rem}
  An immediate consequence of \eqref{eq:25},  Jensen's inequality
  for conditional  expectations and the fact that
  $\mid \operatorname{\mathcal{V}}_r(X) \mid\le 1$ is the inequality 
  \begin{equation}
    \label{eq:28}
  J_{\mathrm{st}, \gamma(\alpha, \lambda)}(X_r)  \le \lambda r E[X]
  \end{equation}
for all $\lambda, r \ge 0$ and all $\alpha\ge 1/2$. 
\end{rem}
\begin{proof}
  Identity \eqref{eq:25} follows immediately from \eqref{eq:38} and
  \eqref{eq:37}.  To see \eqref{eq:38} note how for all smooth test
  functions 
  \begin{align}
    & E \left[ \left( \rho_r^{\gamma}(X_r) + \lambda \sqrt{X_r}
      -(\alpha- 1/2)/\sqrt{X_r} \right)
    \sqrt{X_r} \phi(X_r) \right] \nonumber\\
    & = E \left[\rho_r^{\gamma}(X_r)\sqrt{X_r} \phi(X_r)    \right] +
      \lambda E \left[ X_r \phi(X_r) \right] - E \left[  (\alpha- 1/2)
       \phi(X_r)\right]\nonumber \\
    & = - E \left[ \sqrt{X_r} \left( \frac{1}{2 \sqrt{X_r}} \phi(X_r)
      + \sqrt{X_r} \phi'(X_r) \right) \right]  + \lambda E \left[ X_r
      \phi(X_r) \right] - E \left[  (\alpha- 1/2)
       \phi(X_r)\right]\nonumber\\
& = -  \alpha E \left[ \phi(X_r)\right] 
      -  E \left[ {X_r} \phi'(X_r) \right]  + \lambda E \left[ X_r
      \phi(X_r) \right].\label{eq:39}
  \end{align}
Expanding  \eqref{eq:34}  we can rewrite the third summand
as 
\begin{align*}
 \lambda  E \left[ X_r \phi(X_r) \right] & =E \left[ \lambda \gamma(\alpha-
                                           \frac{1}{2}, \lambda)
       \phi(X_r)\right] + E \left[ \lambda rX \phi(X_r) \right] +
                                    \sqrt{2\lambda}
                                   E \left[  \sqrt{rX} N \phi(X_r)  \right]
+ \frac{1}{2} E \left[ N^2\phi(X_r) \right]. 
\end{align*}
Applying \eqref{eq:29}  to the function $ \gamma \mapsto 
\phi(\gamma + Y_r^2)$ we get 
\begin{align}
  \label{eq:51}
 E \left[ \lambda \gamma(\alpha- \frac{1}{2}, \lambda)\phi(X_r)\right]
  & = (\alpha-\frac{1}{2}) E \left[ \phi(X_r) \right] + E \left[
    \gamma(\alpha-\frac{1}{2}, \lambda) \phi'(X_r) \right]. 
\end{align}
Applying \eqref{eq:1} to the
function $n \mapsto   n\phi(\gamma(\alpha-
                                           \frac{1}{2}, \lambda)+(\sqrt{r X} +
  {n}/{\sqrt{2\lambda}})^2)$ we get 
\begin{align}
  \frac{1}{2} E \left[ N^2\phi(X_r) \right] % & = E \left[ N (N\phi(\sqrt{r X} +
  % {N}/{\sqrt{2\lambda}})) \right] \\
  & = \frac{1}{2}E \left[ \phi(X_r)  \right] + E \left[ N \phi'(X_r)  \frac{
    Y_r}{\sqrt{2 \lambda}} \right].\label{eq:48}
\end{align}
Applying \eqref{eq:1} to $n \mapsto   \sqrt{rX}  \phi(\gamma(\alpha-
                                           \frac{1}{2}, \lambda)+(\sqrt{r X} +
  {n}/{\sqrt{2\lambda}})^2)$ we get 
\begin{align}\label{eq:55}
   \sqrt{2\lambda} E \left[  \sqrt{rX} N \phi(X_r)  \right] & 
=  E \left[  \sqrt{rX} \phi'(X_r) {
    2 Y_r} \right]. 
\end{align}
Resuming from \eqref{eq:39} we compute
\begin{align}
&    E \left[ \left( \rho_r^{\gamma}(X_r) + \lambda \sqrt{X_r}  -(\alpha- 1/2)/\sqrt{X_r} \right)
    \sqrt{X_r} \phi(X_r) \right] \nonumber\\
% &  =  -  E \left[ {X_r} \phi'(X_r) \right] + E \left[ \lambda rX \phi(X_r)
%                                    \right] + E  \left[  \sqrt{rX} \phi'(X_r) {
%     2 Y_r} \right] + E \left[ N \phi'(X_r)  \frac{
%     Y_r}{\sqrt{2 \lambda}} \right] \\
& =   E \left[ \left\{ -\alpha +(\alpha-\frac{1}{2})  +\frac{1}{2}  + \lambda rX  \right\}\phi(X_r)
  \right]  + E  \left[  \left\{ -X_r+ \gamma(\alpha-\frac{1}{2},
  \lambda) + 
  \sqrt{rX}  {
  2 Y_r} + N  \frac{
  Y_r}{\sqrt{2 \lambda}}  \right\} \phi'(X_r) \right] \nonumber\\
& =   E \left[    \lambda rX  \phi(X_r)
  \right] 
% & \quad + E  \left[  \left\{ -X_r+
%     \sqrt{rX}  {
%     2 Y_r} + N  \frac{
%     Y_r}{\sqrt{2 \lambda}}  \right\} \phi'(X_r) \right] \\
%   & =   E \left[ \lambda rX \phi(X_r)
%                                    \right] + E  \left[  (-Y_r+
%    2  \sqrt{rX}    +  \frac{
%     N}{\sqrt{2 \lambda}}) Y_r \phi'(X_r) \right] \\
%   & =  E \left[ \lambda rX \phi(X_r)
%                                    \right] +
+     
 E  \left[   \sqrt{rX} {Y_r} \phi'(X_r)
    \right]\nonumber \\
   & =  \lambda E \left[  \sqrt{rX} \left( \sqrt{rX} +
     \frac{N}{\sqrt{2\lambda}}\right)  \phi(X_r)
     \right],  \label{eq:24}
\end{align}
the last identity being a consequence of \eqref{eq:55}. By a standard
density argument (identity \eqref{eq:24} is valid for all smooth
functions with compact support) we can then deduce the representation
\begin{equation}\label{eq:53}
  \rho_r^{\gamma}(X_r) + \lambda \sqrt{X_r}  -(\alpha- 1/2)/\sqrt{X_r}
  =  \frac{\lambda E \left[  \sqrt{rX} \left( \sqrt{rX} +
        \frac{N}{\sqrt{2\lambda}}\right)  \, | \, X_r\right]}{\sqrt{X_r}}
  = \lambda E \left[ \sqrt{rX}  \operatorname{\mathcal{V}}_r(X)  \, | \,
    X_r \right]
\end{equation}
which leads to \eqref{eq:38}. 
\end{proof}

Combining \eqref{eq:25} with the gamma-specific De Bruijn identity
\eqref{eq:33} we  immediately obtain that 
if $X$ is a positive random variable with finite $\alpha+4$ moment
then 
  \begin{equation}
    \label{eq:57}
    \frac{\mathrm{d}}{\mathrm{dr}} D(X_r \, || \, \gamma(\alpha,
    \lambda/(1+r))) =  \frac{ \lambda}{r} E \left[ E
        \left[ \sqrt{rX} \operatorname{\mathcal{V}}_r(X)  \, | \, X_r
        \right]^2 \right]  -  \alpha \frac{r}{r+1}
  \end{equation}
for all $r>0$. 

\begin{exa}
  Suppose that the input signal $X$ is gamma distributed with
  parameters $(\alpha, \lambda)$ so that $X_r$ follows a gamma law
  with parameters $(\alpha, \lambda/(1+r))$ for each $r>0$ (recall
  Remark~\ref{rem:gammagamma}). Then, thanks to (\ref{eq:53}), we have
\begin{equation}
  \label{eq:43}
  E \left[ \sqrt{r X}
    \operatorname{\mathcal{V}}_r(X)\, | \, X_r \right]  =
\frac{r}{1+r}  \sqrt{X_r}
\end{equation}
so that  $  J_{\mathrm{st}, \gamma(\alpha, \lambda)} (X_r) = \frac{\alpha
    r^2}{r+1} $ 
and $ \frac{\mathrm{d}}{\mathrm{dr}} D(X_r \, || \, \gamma(\alpha,
    \lambda/(1+r))) = 0$, 
as expected.
\end{exa}

%\begin{exa}
  %Same as above, this time with binary input.
%\end{exa}

\section{Mutual information and estimation theory}
\label{sec:form-mutu-inform}

%We derive an analytical
% proof of this fact through a simple link between mutual information
% and relative entropy. This proof combines ideas contained in Section
% II-D of \cite{GuoShamVer051} together with a relation similar to
% formula (65) of \cite{GuoShamVer051}. The method behind this proof is
% robust towards a change of channel, in the sense that it is readily
% transposed from the additive Gaussian channel to the quadratic gamma
% channel.

We start by restating identity \eqref{MI-RE-gau} (which actually holds
true for any channels) in the present gamma-target context. Let $r>0$,
$\tau(r)=r/(r+1)$ and introduce the random variables
\begin{align}
  &\tilde{X}_{\tau(r)}(x)=(1-\tau(r))\gamma({\alpha-\frac{1}{2},\lambda})+\big(\sqrt{\tau(r)x}+\dfrac{\sqrt{1-\tau(r)}}{\sqrt{2\lambda}}Z\big)^2,\\
& X_r(x) = \gamma(\alpha, 1/2) + \left( \sqrt{rx} + \frac{N}{\sqrt{2\lambda}} \right)^2
\end{align}
for $x \in [0, +\infty)$. We also write $ X_r =  X_r(X) $ and $\tilde{X}_{\tau(r)} =
\tilde{X}_{\tau(r)}(X)$. Then
$I(X,{X}_{r}) = I(X,\tilde{X}_{\tau(r)})$ and
\begin{equation}
\label{MI-RE}
I(X,\tilde{X}_{\tau(r)})= E_X \left[
  D(\tilde{X}_{\tau(r)}(X)\mid\mid \gamma(\alpha, \lambda))
\right]-D(\tilde{X}_{\tau(r)}\mid\mid \gamma(\alpha, \lambda)),
\end{equation}
where $E_X[\cdot]$ denotes an expectation taken with respect to $X$ as in (\ref{ExpectationX}).   
Similarly as in Section \ref{sec:ident-gauss-targ} we also deduce from
the gamma-De Bruijn identity (see Theorem 14 of \cite{ArSw1}) as well as the chain rule
for differentiation:
\begin{lem}\label{DB-MMSE} Let $\alpha\geq 1/2$, $\lambda>0$ and $r>0$.
  If $X$ is almost surely positive with finite $\alpha+4$ moments and
  mean $E[X] = \alpha/\lambda$ then
\begin{equation}
\label{DB-MMSE-eq}
\dfrac{d}{dr}\bigg(I(X,X_{r})\bigg)=\dfrac{1}{r(r+1)}\bigg(E_X
\left[ J_{st, \gamma(\alpha, \lambda)}(\tilde{X}_{\tau(r)}(X)) \right]-J_{st,\gamma(\alpha, \lambda)}(\tilde{X}_{\tau(r)})\bigg),
\end{equation}
% with,
% \begin{align*}
% &J^{\frac{1}{2}}_{st}(Y)=\frac{1}{2}E[(\rho_Y(Y)+Y)^2],\\
% &J^{1}_{st}(Y)=\frac{1}{\lambda}E[Y(\rho_Y(Y)+\lambda-\dfrac{\alpha-1}{Y})^2].
% \end{align*}
\end{lem}
 
We are now in a position to obtain the gamma counterpart to the GSV
identity \eqref{eq:MMSE}. However, as already pointed out, the problem
with the quadratic gamma channel is that it is more difficult to
compute directly $J_{st, \gamma(\alpha, \lambda)}(\tilde{X}_{\tau(r)}(x))$ because
$\tilde{X}_{\tau(r)}(x)$ is not a gamma random variable but rather a
non-central gamma whose explicit density is complicated to manipulate.

\begin{prop}\label{Boundalphademi}
Let $\alpha\geq 1/2$, $\lambda>0$, $r>0$ and $X$ be a positive random variable with finite $\alpha+4$ moment
and mean equal to $\alpha/\lambda$. Then 
\begin{equation}
\label{gammaGSV}
\dfrac{d}{dr}\bigg(I\big(X,X_{r}\big)\bigg)=
\lambda\bigg(E_X \left[ X   E \big[ E \big[
  \operatorname{\mathcal{V}}_r(X) \, | \, X_{r}(X) \big]^2 \big]
\right]- E\big[E[\sqrt{X}\operatorname{\mathcal{V}}_r(X) \mid X_{r}]^2\big]\bigg).
\end{equation}
% \begin{equation}
% \label{gammaGSV}
% \dfrac{d}{dr}\bigg(I\big(X,X_{r}\big)\bigg)=\lambda\bigg(\int_\mathbb{R^*_+} x E \big[ E \big[ \dfrac{Y^X_r}{\sqrt{X^X_r}} \, | \, X^X_r \big]^2 \big]p_X(x)dx-E\big[E[\sqrt{ X}\frac{Y_{r,1}}{\sqrt{X_r}} \mid X_r]^2\big]\bigg).
% \end{equation}
\end{prop}
\begin{proof}
First of all, applying Lemma \ref{DB-MMSE}, we have:
\begin{align}\label{1-step}
\dfrac{d}{dr}\bigg(I\big(X,X_{r}\big)\bigg)=\dfrac{1}{r(r+1)}\bigg(\int
  p_X(x)J_{st,  \gamma(\alpha,
  \lambda)}(\tilde{X}_{\tau(r)}(x))dx-J_{st, \gamma(\alpha, \lambda)}(\tilde{X}_{\tau(r)})\bigg),
\end{align}
Applying a slight extension of \eqref{eq:12} to $\tilde{X}_{\tau(r)} = \frac{1}{r+1}X_r$ we
deduce
\begin{equation}
\label{scaling1}
J_{st, \gamma(\alpha, \lambda)}(\tilde{X}_{\tau(r)})=(1+r)J_{st,
  \gamma(\alpha, \lambda)}(X_{r})-\alpha r^2.
\end{equation}
Applying Proposition \ref{prop:interr-gamma-mmse-1} then leads to 
\begin{equation}
\label{scaling2}
J_{st, \gamma(\alpha,\lambda)}(\tilde{X}_{\tau(r)})=(1+r)\lambda E\big[E[\sqrt{r
  X}\operatorname{\mathcal{V}}_r(X) \mid X_{r}]^2\big]-\alpha r^2.
\end{equation}
\noindent
Now, Proposition \ref{prop:interr-gamma-mmse-1} is true irrespectively
of the distribution of the input, so that we also have (for each fixed
$x$) 
\begin{align*}
J_{st, \gamma(\alpha \lambda)}(X_{r}(x)) = \lambda E \big[ E \big[ \sqrt{r x}
    \operatorname{\mathcal{V}}_r(x) \, | \, X_{r}(x) \big]^2 \big].
\end{align*}
Furthermore, we have % the following relationship between $J_{st,
  % \gamma(\alpha, \lambda)}(\tilde{X}_{\tau(r)}(x))$ and $J_{st,
  % \gamma(\alpha, \lambda)}(X_{r}(x))$ (note that $x\ne \alpha/\lambda$):
\begin{equation}
J_{st, \gamma(\alpha, \lambda)}(\tilde{X}_{\tau(r)}(x))=(1+r)J_{st,\gamma(\alpha, \lambda)}(X_{r}(x))-2r^2\lambda x+\dfrac{\lambda r^2}{(1+r)}\big(\frac{\alpha}{\lambda}+rx\big),
\end{equation}
which leads to
\begin{align}
J_{st,\gamma(\alpha, \lambda)}(\tilde{X}_{\tau(r)}(x))=(1+r)\lambda E
  \big[ E \big[ \sqrt{r x}\operatorname{\mathcal{V}}_r(x)  \, | \,
  X_{r}(x)  \big]^2 \big]-2r^2\lambda x+\dfrac{\lambda r^2}{(1+r)}\big(\frac{\alpha}{\lambda}+rx\big).
\end{align}
Integrating the previous expression with respect to the density of $X$
together with the fact that $E[X]=\frac{\alpha}{\lambda}$, we obtain 
\begin{align}
\int_\mathbb{R^*_+}J_{st, \gamma(\alpha,
  \lambda)}(\tilde{X}_{\tau(r)}(x))p_X(x)dx&=(1+r)\int_\mathbb{R^*_+}\lambda
                                             E \big[ E \big[ \sqrt{r
                                             x}\operatorname{\mathcal{V}}_r(x)
                                             \, | \, X_{r}(x) \big]^2 \big]p_X(x)dx\\
&-2r^2\lambda \int_{\mathbb{R}^*_+}xp_X(x)dx+\dfrac{\lambda r^2}{1+r}\big(\frac{\alpha}{\lambda}+r\int_\mathbb{R^*_+}x p_X(x)dx\big),\\
&=(1+r)\int_\mathbb{R^*_+}\lambda E \big[ E \big[ \sqrt{r x}\operatorname{\mathcal{V}}_r(x) \, | \, X_{r}(x) \big]^2 \big]p_X(x)dx-2r^2\alpha+\alpha r^2,\\
&=(1+r)\int_\mathbb{R^*_+}\lambda E \big[ E \big[ \sqrt{r x}\operatorname{\mathcal{V}}_r(x) \, | \, X_{r}(x) \big]^2 \big]p_X(x)dx-\alpha r^2.\label{eq*2}
\end{align}
Combining (\ref{eq*2}) and (\ref{scaling2}) together with (\ref{1-step}), we obtain the relation:
\begin{equation}
\dfrac{d}{dr}\bigg(I\big(X,X_{r}\big)\bigg)=\lambda\bigg(\int_\mathbb{R^*_+}
x E \big[ E \big[ \operatorname{\mathcal{V}}_r(x) \, | \, X_{r}(x)  \big]^2 \big]p_X(x)dx-E\big[E[\sqrt{ X}\operatorname{\mathcal{V}}_r(X) \mid X_{r}]^2\big]\bigg)
\end{equation}
leading directly to the claim.
% Using the simple fact that $\mid Y^X_r/\sqrt{X^X_r} \mid \leq 1$, we readily obtain the bound (thanks to Jensen inequality):
% \begin{align}
% \dfrac{d}{dr}\bigg(I\big(X,X_r\big)\bigg)\leq \lambda\bigg(E[X]-E\big[E[\sqrt{ X}\frac{Y_{r,1}}{\sqrt{X_r}} \mid X_r]^2\big]\bigg).
% \end{align}

%Moreover, recall the following formula for the density of $\tilde{X}_{\tau(r),1,x}$ (true for any $\alpha$, see \cite{ArSw1}):
%\begin{align}\label{density}
%g^{(\alpha,\lambda)}(\tau(r),u,x)=&\dfrac{\lambda}{1-\tau(r)}\bigg(\dfrac{u}{x\tau(r)}\bigg)^{\frac{\alpha-1}{2}}\exp\bigg(-\frac{\lambda u}{1-\tau(r)}\bigg)\nonumber\\
%&\times\exp\bigg(-\dfrac{\lambda x \tau(r)}{1-\tau(r)}\bigg)I_{\alpha-1}\bigg(\dfrac{2\lambda\sqrt{ux\tau(r)}}{1-\tau(r)}\bigg).
%\end{align}
%By straightforward but tedious computations, we have for $J^{1}_{st}(\tilde{X}_{\tau(r),1,x})$:
%\begin{align}
%J^{1}_{st}(\tilde{X}_{\tau(r),1,x})=\int_0^{+\infty}g^{(\frac{1}{2},\lambda)}(\tau(r),u,x)&\lambda^2\bigg[r^2u-2r(1+r)\sqrt{\dfrac{rux}{1+r}}\tanh\bigg(2\lambda(1+r)\sqrt{\dfrac{rux}{1+r}}\bigg)\nonumber\\
%&+(1+r)rx \tanh^2\bigg(2\lambda(1+r)\sqrt{\dfrac{rux}{1+r}}\bigg)\bigg]du.
%\end{align}
%Using the obvious fact that $0\leq \tanh(z)\leq 1$ for all $z>0$, we obtain the bound:
%\begin{align}
%J^{1}_{st}(\tilde{X}_{\tau(r),1,x})&\leq \lambda^2 r^2E[\tilde{X}_{\tau(r),1,x}]+\lambda^2(1+r)rx\nonumber\\
%&\leq \frac{\lambda r^2 (1+2rx\lambda)}{2(1+r)}+\lambda^2(1+r)rx.
%\end{align}
%To conclude we combine the previous bound together with (\ref{1-step}) and (\ref{scaling2}) and the fact that $E[X]=1/(2\lambda)$.

\end{proof}
\begin{rem}
   It should be clear that the previous result holds true even if $E[X]\ne \alpha/\lambda$. The proof is similar by using the general relation,
\begin{equation}
J_{st, \gamma(\alpha, \lambda)}(\tilde{X}_{\tau(r)})=(1+r)J_{st,
  \gamma(\alpha,\lambda)}(X_r)-2r^2\lambda E[X]+\dfrac{\lambda
  r^2}{(1+r)}\big(\frac{\alpha}{\lambda}+rE[X]\big), 
\end{equation}
instead of (\ref{scaling1}).

\end{rem}
\subsection{An upper bound}

 An immediate consequence of \eqref{gammaGSV} and the fact that
  $\mid\operatorname{\mathcal{V}}_r(X)\mid \le 1$ is the upper bound 
\begin{equation}\label{bound1}
  \dfrac{d}{dr}\bigg(I\big(X,X_{r}\big)\bigg)\leq
  \lambda\bigg(E[X]-E\big[E[\sqrt{ X} \operatorname{\mathcal{V}}_r(X)\mid X_{r}]^2\big]\bigg).
\end{equation}
Note that (\ref{bound1}) is very close to the Gaussian GSV identity
(\ref{eq:MMSE}).  In particular when $X\sim \gamma_{\alpha,\lambda}$,
using (\ref{eq:43}) and \eqref{bound1}, we have:
\begin{align}
\dfrac{d}{dr}\bigg(I\big(X,X_{r}\big)\bigg)\leq \dfrac{\alpha}{1+r},
\end{align}
which leads to the fine bound:
\begin{equation}\label{bound2}
I\big(X,X_r\big)\leq \alpha \log\big(1+r\big).
\end{equation}
The previous bound should be compared with the corresponding formula
$(11)$ of \cite{GuoShamVer051} which is satisfied by the mutual
information of the additive Gaussian channel with Gaussian input. When
$X\sim \gamma_{\alpha,\nu}$, we have the bound:
\begin{align}
I\big(X,X_{r}\big)\leq \alpha \log\big(1+\frac{\lambda r}{\nu}\big).
\end{align}

\subsection{A lower bound for $\alpha=\frac{1}{2}$}
We set $\alpha=\frac{1}{2}$. Assume that the input is gamma distributed with parameters $(1/2,\lambda)$. By definition, the mutual information between $X$ and $X_r$ is equal to:
\begin{align*}
I\big(X,X_r\big)=\int p_{X_r\mid X=x}\big(u,x\big)p_X(x)\log\bigg(\dfrac{p_{X_r\mid X=x}\big(u,x\big)}{p_{X_r}(u)}\bigg)dudx
\end{align*}
Let us compute explicitly the ratio between $p_{X_r\mid X=x}\big(u,x\big)$ and $p_{X_r}(u)$ in order to provide a simple lower bound for the logarithm term in the previous expression. We have:
\begin{align*}
\dfrac{p_{X_r\mid X=x}\big(u,x\big)}{p_{X_r}(u)}&=\Gamma(\alpha)\dfrac{\lambda e^{-\lambda u}e^{-\lambda r x}(\frac{u}{rx})^{\frac{\alpha-1}{2}}I_{\alpha-1}\big(2\lambda \sqrt{u x r}\big)}{u^{\alpha-1}e^{-\frac{\lambda u}{r+1}}\big(\frac{\lambda}{r+1}\big)^\alpha},\\
&=\Gamma(\alpha)(1+r)^\alpha \frac{1}{\lambda^{\alpha-1}}e^{-\frac{\lambda r u}{r+1}}e^{-\lambda r x}\dfrac{I_{\alpha-1}\big(2\lambda \sqrt{u x r}\big)}{\big(u r x\big)^{\frac{\alpha-1}{2}}}.
\end{align*}
Moreover, since $I_{-1/2}(z)=\sqrt{2/\pi}\cosh\big(z\big)/\sqrt{z}$, we obtain:
\begin{align}
\dfrac{p_{X_r\mid X=x}\big(u,x\big)}{p_{X_r}(u)}&=\sqrt{\pi}\sqrt{(1+r)} \frac{1}{\lambda^{-\frac{1}{2}}}e^{-\frac{\lambda r u}{r+1}}e^{-\lambda r x}\dfrac{\cosh\big(2\lambda\sqrt{u x r}\big)}{\sqrt{\pi \lambda}},\\
&=\sqrt{(1+r)}e^{-\frac{\lambda r u}{r+1}}e^{-\lambda r x}\cosh\big(2\lambda\sqrt{u x r}\big),\\
&\geq \frac{1}{2}\sqrt{(1+r)}e^{-\frac{\lambda r u}{r+1}}e^{-\lambda r x} e^{2\lambda\sqrt{u x r}}.
\end{align}
Using the monotonicity of the logarithm, we obtain:
\begin{align}
\log\bigg(\dfrac{p_{X_r\mid X=x}\big(u,x\big)}{p_{X_r}(u)}\bigg)\geq \log\bigg(\frac{1}{2}\sqrt{(1+r)}e^{-\frac{\lambda r u}{r+1}}e^{-\lambda r x} e^{2\lambda\sqrt{u x r}}\bigg).
\end{align}
This inequality implies the following on the mutual information between $X$ and $X_r$:
\begin{align}
I\big(X,X_r\big)&\geq \frac{1}{2}\log\big(1+r\big)-\log(2)-\frac{\lambda r}{r+1}E[X_r]-r\lambda E[X]+2 \lambda \sqrt{r}E[\sqrt{X_r}\sqrt{X}],\\
&\geq \frac{1}{2}\log\big(1+r\big)-\log(2)-r\alpha -r\alpha+2\lambda \sqrt{r}E[\mid\sqrt{r}X+\frac{Z\sqrt{X}}{\sqrt{2\lambda}} \mid ],\\
&\geq \frac{1}{2}\log\big(1+r\big)-\log(2)-2r\alpha+2\alpha r=\frac{1}{2}\log\big(1+r\big)-\log(2),
\end{align}
where we have used the fact $\mid x\mid\geq x$, $X$ and $Z$ are independent and $E[Z]=0$. This lower bound combined with the bound (\ref{bound2}) implies that:
\begin{align}\label{eq:3}
\underset{r\rightarrow+\infty}{\lim}\dfrac{I\big(X,X_r\big)}{\frac{1}{2}\log(1+r)}=1
\end{align}
\begin{rem}\label{highSNR}
\begin{itemize}
\item Thus, for $\alpha=1/2$ and for a gamma-$(1/2,\lambda)$ distributed input, the mutual information between $X$ and the output $X_r$ exhibits the same asymptotic for large values of the channel quality parameter $r$ as the mutual information between the additive Gaussian channel and a Gaussian input.
\item It would be nice to know if such an asymptotic is still true for
  $\alpha>1/2$ and a gamma-$(\alpha,\lambda)$ distributed input. More
  generally we ask the question: for which input distribution do we
  have the same type of asymptotic as in \eqref{eq:3} for the mutual
  information ? Such questions are related to the concept of ``MMSE
  dimension'', see \cite{wuverdu}.
\end{itemize}
\end{rem}

\section*{Acknowledgements} 
The research of BA is funded by a Welcome Grant from Liege University.

%%%%%%%%%%%%%%%%%%%%%%%%%%%%%%%%%%%%%%%%%%%%%%%%%%%%%%%%%%%%%%%%%%%%%%%%%%%%%%%%%%%%%%%%%%%%%%%%%%%%%%%%%%%%%%%%%%

\def\cprime{$'$}

\end{document}